\documentclass[11pt]{amsart}
\usepackage[all]{xy}
\usepackage{mathrsfs}

\usepackage[margin=1in]{geometry}
\usepackage{graphicx}
\usepackage{aeguill}
\usepackage[french,english]{babel}
\usepackage{amsmath,amsfonts,amssymb, lineno}
\usepackage{color}
\usepackage[utf8]{inputenc}
\usepackage{comment}
\usepackage[shortlabels]{enumitem}
\usepackage{etoolbox}
\usepackage{float}
\usepackage{latexsym}
\usepackage{lipsum}
\usepackage{needspace}
\usepackage{tikz}
\usepackage[colorlinks=true,linkcolor=blue,citecolor=purple,urlcolor=blue,backref=page,pagebackref=true]{hyperref}

\newtheorem{theorem}{Theorem}
\numberwithin{theorem}{section}
\numberwithin{equation}{section}

\newtheorem{lemma}[theorem]{Lemma}

\newtheorem{proposition}[theorem]{Proposition}
\newtheorem{corollary}[theorem]{Corollary}
\theoremstyle{definition}
\newtheorem{definition}[theorem]{Definition}
\theoremstyle{remark}
\newtheorem{remark}[theorem]{Remark}
\newtheorem{example}[theorem]{Example}

\newcommand{\Z}{\mathbb{Z}}
\newcommand{\R}{\mathbb{R}}
\newcommand{\N}{\mathbb{N}}

\newcommand{\conv}{\operatorname{conv}}
\newcommand{\aff}{\operatorname{aff}}
\newcommand{\dist}{\operatorname{dist}}
\newcommand{\vol}{\operatorname{vol}}

\begin{document}

\title{Lattice point sumsets and asymptotic approximate groups}
\author{Arindam Biswas}
\address{Polynom Research, Paris, France}
\email{arin.math@gmail.com}
\author{Pavlo Yatsyna}
\address{Charles University \\ Faculty of Mathematics and Physics \\ Department of Algebra \\ Sokolovsk\'{a} 83 \\ 186 75 Praha 8 \\ Czech Republic}
\email{p.yatsyna@matfyz.cuni.cz}
\date{\today}

\vspace{1em}  

\vspace{1em}

\begin{abstract}
We establish new quantitative bounds for asymptotic approximate groups arising from finite subsets of lattices and, more generally, semi-linear subsets of abelian groups. Our approach combines Khovanskii's theorem on sumsets with Rogers and Zong bounds for the covering numbers.
\end{abstract}

\maketitle

\section{Introduction}

Let $G$ be a group. For non-empty subsets $A_1,\ldots,A_m\subseteq G$, we write
\[
A_1\cdot A_2\cdots  A_m:=\left\{a_1\cdot a_2\cdots  a_m:a_i\in A_i\text{ for }1\le i\le m\right\}
\]
in particular, for a subset $A\subseteq G$, the \textit{$h$-fold} product set is denoted by
\[
A^h:=\left\{a_1\cdot\cdots\cdot a_h:a_1,\ldots,a_h\in A\right\}.
\]

\textit{Approximate groups} provide a way of formalising the idea that a finite set behaves like a group up to a bounded error under multiplication. The formal definition of an approximate group was introduced by Tao in \cite{Tao08} and a part of it was motivated by its use in the work of Bourgain-Gamburd \cite{BG08} on super-strong approximation for Zariski-dense subgroups of $\mathrm{SL}_2(\mathbb{Z})$. Tao's definition of a $K$-approximate group requires a finite set to be symmetric, to contain the identity, and to have its square covered by at most $K$ left-translates of itself.

\begin{definition}
\label{AppTao}
Let $G$ be a group, and let $K\ge 1$. A finite set $A\subseteq G$ is called a \textit{$K$-approximate group} if
\begin{enumerate}
\item $e\in A$, where $e$ is the identity element of $G$;
\item $A$ is symmetric, that is, $a^{-1}\in A$ whenever $a\in A$;
\item there exists a symmetric set $X\subseteq A\cdot A$ with $|X|\le K$ such that
\[
A\cdot A\subseteq X\cdot A.
\]
\end{enumerate}
\end{definition}

This definition is now a standard object in additive combinatorics and has deep connections with growth in groups, beginning with the influence of Freiman-type inverse problems and including applications to expansion and approximate subgroup structure. See \cite{BGT11, Breuillard14}. Nathanson introduced a broader notion in which no symmetry, finiteness, or identity condition is imposed on the set being studied \cite{Nathanson2018}.

\begin{definition}
\label{def:appgr}
Let $r,\ell\in\N$ and $G$ be a group. A non-empty subset $A\subseteq G$ is an \textit{$(r,\ell)$-approximate group} if there exists a set $X\subseteq G$ such that
\[
|X|\le \ell
\qquad\text{and}\qquad
A^r\subseteq X\cdot A.
\]
\end{definition}

Thus, every non-empty subset is trivially a $(1,1)$-approximate group. Moreover, every $K$-approximate group in the sense of Definition~\ref{AppTao} is a $(2,K)$-approximate group in the sense of Definition~\ref{def:appgr}. One interest of Definition~\ref{def:appgr} is not in a single product set alone, but in the behaviour of large powers of a fixed set \cite{Nathanson2018}.

\begin{definition}
\label{def:aag}
Let $r,\ell\in\N$. A subset $A\subseteq G$ is an \textit{asymptotic $(r,\ell)$-approximate group} if there exists $h_0\in\N$ such that, for each $h\ge h_0$, $A^h$ is an $(r,\ell)$ approximate group. That is, there exists $X\subseteq G$ (depending on $h$) such that $|X|\le \ell$ and $A^{rh}\subseteq X\cdot A^h$.
\end{definition}

In the case when $G$ is an abelian group, we replace multiplication by addition, and work with sumsets instead, i.e. $A^r\subseteq X\cdot A$ becomes $rA\subseteq X+A$, where $A+X=\{a+x: a\in A,\,x\in X\}$.

Nathanson proved that there exist non-abelian groups that contain a subset that is not an asymptotic approximate group for all $r\ge 2$ and $\ell \ge 1$ \cite[Theorem~1]{Nathanson2018}. Complementarily, he proved that every non-empty finite subset of an abelian group is an asymptotic approximate group.

\begin{theorem}[Nathanson \cite{Nathanson2018}]
\label{ThmNathanson}
Let $r,k\in\N$, and let $A$ be a finite subset of cardinality $k$ in an abelian group. Then $A$ is an asymptotic $(r,\ell)$-approximate group for some $\ell\in\N$.
\end{theorem}

More quantitatively, Nathanson \cite[Theorem~5]{Nathanson2018} obtained a bound of the form
\[
\ell\le n_0 k b(r,k),
\qquad
b(r,k):=\binom{(r+1)(k-1)-1}{k-1},
\]
where $n_0$ is the cardinality of the torsion subgroup of $\langle A\rangle$, the group generated by $A$. In the special one-dimensional case, Nathanson had earlier proved the sharper result that every finite set of integers is an asymptotic $(r,r+1)$-approximate group \cite{Nathanson2016}. Biswas and Moens \cite[Theorem 1.6]{BiswasMoens2022102330} later gave a different proof in the general abelian setting, improving the bound to $\ell\le (4rk)^k.$ More recently, a chromatic version of the theory was developed for tuples of subsets of abelian groups \cite{Biswas2026}.

Our main result is toward a quantitative improvement of the above constant for sufficiently large $k$ for sumsets of lattice points. Specifically:

\begin{theorem}\label{thm:main} 
Let $A\subseteq \Z^n$ be a non-empty finite set with affine span of dimension $d$. For any integer $r\ge 1$ and any real $\delta>0$, $A$ is an asymptotic $(r,\ell)$-approximate group with $\ell = 1$ if $d=0$ and otherwise \[\ell \le (1+\delta)C(d,r)\vartheta(d),
\]
where $\vartheta(d)$ is the translative covering density in $\R^d$, which satisfies $\vartheta(1)=1$, $\vartheta(2)\le \frac{3}{2}$, and $\vartheta(d)\le d\log d +d\log\log d+5d$ for $d\ge 3$. The constant $C(d,r)$ is given by:
    \[
    C(d,r) = \begin{cases}
    1+r & \text{if } d = 1,\\
    1+4r+r^2 & \text{if } d = 2,\\
    1+9r+9r^2+r^3 & \text{if } d = 3,\\
    (1+\sqrt{r})^{2d}\sqrt{\frac{\pi d}{2}}e^{1/12} & \text{if } d \ge 4.
    \end{cases}
    \]
    If $\conv(A)$ is centrally symmetric, we can replace $C(d,r)$ with $(1+r)^d$ for all $d\ge 1$.
\end{theorem}

We note that all constants are effective, due to the recent work of \cite{GranvilleShakanWalker2023} and \cite{GranvilleSmithWalker2025}. The previous quantitative results depended on $k=|A|$, and, in Nathanson's general abelian-group result, also on the torsion subgroup of $\langle A\rangle$. To improve these bounds in the general case where $A$ is of large cardinality sitting in a relatively smaller dimensional space, one needs to incorporate finer convex-geometric features of the lattice configuration, such as the affine dimension of $\operatorname{conv}(A)$. Instead of measuring the complexity of $A$ by its cardinality, we measure it by the affine dimension $d$ of its convex hull, and incorporate a covering constant depending only on $r$, $d$, and the translative covering density in $\R^d$, up to an arbitrarily small multiplicative loss. In particular, for fixed $r$ and $d$, the bound is uniform over all finite sets $A\subseteq\Z^n$ of affine dimension $d$, regardless of how large $|A|$ is.

This is the principal quantitative improvement over the preceding finite-set bounds when $A$ has many points but small affine dimension. For example, arbitrarily large finite subsets of a line have an asymptotic covering constant arbitrarily close to $1+r$, whereas the earlier estimates grow with $k=|A|$. More generally, for each fixed $d$, both Nathanson's bound and the Biswas--Moens bound grow rapidly with $k$, while the bound proved here is independent of $k$. When $\operatorname{conv}(A)$ is centrally symmetric, the estimate improves further to the symmetric covering constant $(1+r)^d$, again up to the same arbitrarily small multiplicative loss.

Our method also implies the corresponding statement for arbitrary abelian groups. The additional feature is a finite torsion loss.

\begin{theorem}\label{thm:finite-abelian}
Let $A$ be a non-empty finite subset of an abelian group $G$. Fix $a_0\in A$, let $H:=\langle A-a_0\rangle$, and let $T=H_{\mathrm{tors}}$ be the torsion subgroup of $H$. The quotient $H/T$ is a finitely generated free abelian group, so there exists a group isomorphism $\psi:H/T\xrightarrow{\sim}\Z^n$ for some $n\ge 0$.
Let
\[
\overline{A}=\psi(\{a-a_0+T: a\in A\})\subseteq\Z^n,
\]
and define $d=\dim\aff(\overline{A})$. Then, for every integer $r\ge 1$ and every $\delta>0$, $A$ is an asymptotic $(r,\ell)$-approximate group where
\begin{enumerate}
    \item If $d=0$, one may take $\ell\le |T|$.
    \item If $d\ge 1$, one may take $\ell\le |T|(1+\delta)C(d,r)\vartheta(d)$.
    \item If $\conv(\overline A)$ is centrally symmetric, then $C(d,r)$ may be replaced by $(1+r)^d$.
\end{enumerate}
Furthermore, the affine dimension $d$ and the central symmetry of $\conv(\overline{A})$ are independent of the choice of $a_0\in A$ and $\psi$.
\end{theorem}

Biswas and Moens~\cite{BiswasMoens2022102330} extended the theory of asymptotic approximate groups to \textit{semi-linear sets} -- finite unions of unbounded generalised arithmetic progressions. 

\begin{definition}
A subset $S\subseteq \Z^n$ is called an \emph{additive submonoid}
if
\[
0\in S
\qquad\text{and}\qquad
S+S\subseteq S.
\]
It is called \emph{finitely generated} if there exist
$v_1,\ldots,v_m\in\Z^n$ such that
\[
S=\langle v_1,\ldots,v_m\rangle_{\N}
:=
\left\{
\sum_{j=1}^m c_jv_j : c_j\in\N
\right\}.
\]
A subset $A\subseteq\Z^n$ is called \emph{semi-linear} if it admits a presentation
\[
A=\bigcup_{i=1}^s(a_i+S_i),
\]
where $a_i\in\Z^n$ and each $S_i$ is a finitely generated additive submonoid of $\Z^n$.
\end{definition}

We prove the following:

\begin{theorem}\label{thm:semilinear} 
	Let \[ A=\bigcup_{i=1}^s (a_i+S_i)\subseteq\Z^n \] be a semi-linear set, where each $S_i$ is a finitely generated additive submonoid of $\Z^n$. Put \[ F:=\left\{a_1,\ldots,a_s\right\}, \qquad d_F:=\dim \aff(F). \] Let $r\ge 1$ be an integer and let $\delta>0$. If $d_F=0$, then $A$ is an asymptotic $(r,1)$-approximate group. If $d_F\ge 1$, then there exists $h_0\in\N$ such that, for every $h\ge h_0$, there is a set $X\subseteq\Z^n$ satisfying \[ rhA\subseteq X+hA \] and \[ |X|\le (1+\delta)C(d_F,r)\vartheta(d_F). \] 
	
	Equivalently, $A$ is an asymptotic $(r,\ell)$-approximate group for some integer \[ \ell\le (1+\delta)C(d_F,r)\vartheta(d_F). \] If $\conv(F)$ is centrally symmetric, then one may replace the above cardinality bound by \[ |X|\le (1+\delta)(1+r)^{d_F}\vartheta(d_F). \]
	 \end{theorem}

The proof combines two ingredients from discrete and convex geometry. The first is Khovanskii's theorem on the eventual structure of sumsets of lattice points, which implies that lattice points lying sufficiently far from the boundary of $h\operatorname{conv}(A)$ already belong to $hA$. The second is a Rogers--Zong-type estimate for covering one homothetic copy of a convex body by translates of another. After a rounding step that moves the covering translates into the lattice, the boundary buffer supplied by Khovanskii's theorem absorbs the rounding error. This yields the desired inclusion
\[rhA\subseteq X+hA\]
with $|X|$ bounded only in terms of the affine dimension.

Let $k=|A|$. In the lattice situation considered in Theorem~\ref{thm:main}, Nathanson's torsion factor is equal to $1$, so the earlier finite-set bound is governed by
\[
k\binom{(r+1)(k-1)-1}{k-1}.
\]
The Biswas--Moens estimate gives the uniform abelian-group bound $(4rk)^k$. Both estimates depend on $k$. By contrast, Theorem~\ref{thm:main} gives, for every fixed $\delta>0$, the bound
\[
\ell\le (1+\delta)C(d,r)\vartheta(d),
\qquad d=\dim\operatorname{aff}(A),
\]
which is independent of $k$. Hence the theorem is strongest precisely for finite lattice sets with many points but low-dimensional convex hull. The bound obtained in Theorem~\ref{thm:main} is sharp in its order of growth, see Example~\ref{example:sharpness}.

\section{Preliminaries}\label{sec:prelim}

We restrict ourselves to abelian groups and use additive notation from now on.
Thus, for a subset $A$ of an abelian group, as in the introduction, we write \[ 1A:=A,\qquad 2A:=A+A:=\left\{a_1+a_2:a_1,a_2\in A\right\}, \] and, more generally, for any $h\in \N$ let \[ hA:=\left\{a_1+\cdots+a_h:a_1,\ldots,a_h\in A\right\}, \] 
denote the \textit{h-fold sumset} of $A$.

We introduce the following standard notation from convex and discrete geometry, see Schneider \cite{Schneider}. Throughout, we reserve $d,n\in\N$, $n\ge 1$, and work in $\R^n$ equipped with the standard inner product
$\langle\cdot,\cdot\rangle$ and the Euclidean norm $\|\cdot\|$. For $x\in \R^n$ and real numbers $r,t>0$, denote by \[B(x,r):=\left\{y\in\R^n:\|y-x\|\le r\right\} \text{ and } B(t):=B(0,t)\] the closed
balls of radius $r$ and $t$ centred at $x$ and $0$, respectively. For a non-empty set $K\subseteq \R^n$, its \textit{affine hull} is
\[
\aff(K)
:=
\left\{
\sum_{i=1}^m \lambda_i x_i :
x_i\in K,\ \lambda_i\in\R,\ 
\sum_{i=1}^m \lambda_i=1,\ m\in \N
\right\}.
\]
The \emph{affine dimension} of $K\subseteq \R^n$, denoted $\dim\aff(K)$, is the dimension of the smallest affine subspace of $\R^n$ containing $K$, namely $\aff(K)$.
Equivalently, if $x_0\in K$, then
\[
\dim\aff(K)
=
\dim_{\R} (K-x_0),
\]
where
\[
K-x_0:=\{x-x_0:x\in K\}.
\]
We denote by 
\[\conv(K):=\left\{\sum^m_{i=1}\lambda_i x_i: x_i\in K,\,\lambda_i\ge 0,\,\sum^m_{i=1}\lambda_i=1,\,m\in \N\right\}\]
the \textit{convex hull} of $K$. A set $K\subset\R^n$ is called a \textit{convex body} if it is a compact convex set with non-empty interior. We denote by $\partial K$ the boundary of $K$, i.e. \[\partial K=\left\{x \in K: \forall r>0,\, B(x,r)\not\subseteq K\right\}.\] For non-empty
$K,L\subseteq\R^n$, vectors $v\in\R^n$, and a real number $\lambda\ge 0$:
\begin{itemize}
  \item $K+L=\left\{x+y:x\in K,\,y\in L\right\}$; 
  \item $K+v=\left\{x+v:x\in K\right\}$; 
  \item $K\ominus L:=\left\{x\in\R^n:x+L\subseteq K\right\}$ (\textit{Minkowski difference});
  \item $\lambda K=\left\{\lambda x:x\in K\right\}$;
  \item $\dist(v,K)=\inf_{x \in K}\|v-x\|$.
\end{itemize}
Note that the \textit{difference body} $K-K$ should be considered as $K+(-K)$. 
\begin{lemma}\label{lem:minus}
    Let $K\subseteq \R^n$ be a convex body and $r>0$. Then $K\ominus B(r)=\left\{x\in K:\dist(x,\partial K)\ge r\right\}$.
\end{lemma}
\begin{proof}
    By definition $x\in K\ominus B(r)$ if and only if $B(x,r)\subseteq K$, so it suffices to show that, for a convex body $K$, this is equivalent to $x\in K$ and $\dist(x,\partial K)\ge r$. If $B(x,r)\subseteq K$ then $x\in K$, and no boundary point lies within distance $r$ of $x$ (such a point would be interior), so $\dist(x,\partial K)\ge r$. Conversely, if $x\in K$ with $\dist(x,\partial K)\ge r$ and some $y\in B(x,r)$ had $y\notin K$, then by convexity the segment $[x,y]$ would cross $\partial K$ at a point within distance $\|x-y\|\le r$ of $x$, contradicting $\dist(x,\partial K)\ge r$.
\end{proof}
For a convex body $K\subseteq\R^n$, the \emph{support function} is
\[
  \sigma_K:\R^n\to\R,\qquad \sigma_K(u):=\sup_{x\in K}\langle u,x\rangle.
\]
The following standard properties are collected in Schneider~\cite[\S~1.7]{Schneider}:
for compact convex $K,L\subseteq \R^n$, $u,v\in\R^n$, $\lambda\ge 0$,
\begin{enumerate}[label=\textup{(P\arabic*)}]
  \item $\sigma_{\lambda K}=\lambda \sigma_K$;\label{P1}
  \item $\sigma_{K+v}(u)=\sigma_K(u)+\langle v,u\rangle$;\label{P2}
  \item $\sigma_{K+L}=\sigma_K+\sigma_L$;\label{P3}
  \item $K\subseteq L\iff \sigma_K\le \sigma_L$ pointwise;\label{P4}
  \item $\sigma_{B(t)}(u)=t\|u\|$.\label{P5}
\end{enumerate}

If $B(u,\lambda)\subseteq K$ then
$B(0,\lambda)\subseteq K-u$, and by properties~\ref{P4}
and~\ref{P5},
\begin{equation}\label{eq:inradius}
 \sigma_{K- u}(v)\ge \sigma_{B(\lambda)}(v)= \lambda\|v\| \qquad\text{for all }v\in\R^n.
\end{equation}

\begin{lemma}\label{lem:homothet}
Let $P\subseteq \R^n$ be a convex body, $y\in\R^n$, and $\rho>0$ with $B(y,\rho)\subseteq P$.
For every $t>0$ and $\epsilon\in[0,1]$ satisfying $\epsilon\ge t/\rho$, we have
\[
  (1-\epsilon)(P - y)+y\;\subseteq\;P\ominus B(t).
\]
\end{lemma}

\begin{proof}
By the definition of Minkowski difference, the inclusion is equivalent to
\[
(1-\epsilon)(P-y)+y+B(t)\subseteq P.
\]
Both sets are compact and convex, so by~\ref{P4} it suffices to show
$\sigma_{(1-\epsilon)(P-y)+y+B(t)}\le \sigma_P$. By~\ref{P1}--\ref{P3} and~\ref{P5}:
\[
  \sigma_{(1-\epsilon)(P-y)+y+B(t)}(u)=(1-\epsilon)\sigma_{P- y}(u)+\langle y,u\rangle+t\|u\|,
  \qquad
  \sigma_P(u)=\sigma_{P- y}(u)+\langle y,u\rangle.
\]
Subtracting $\langle y,u\rangle$, the required inequality becomes
$t\|u\|\le\epsilon \sigma_{P- y}(u)$, which follows from~\eqref{eq:inradius}
and $\epsilon\ge t/\rho$.
\end{proof}

Let $\Gamma \subseteq \R^n$. We call $\Gamma$ a \textit{lattice} of \textit{rank} $d$ if there exist $d$ linearly independent vectors $b_1,\ldots,b_d\in \R^n$, i.e., $\Z$-\textit{basis}, such that $\Gamma$ is generated over $\Z$ by these vectors, that is, 

\[\Gamma=\left\langle b_1,\ldots,b_d\right\rangle_\Z := \left\{\sum_{i=1}^d z_ib_i:z_i \in \Z\right\}.\] 

For example, $\Z^n\subseteq \R^n$ is a lattice of rank $n$ generated by the standard basis $e_1,\ldots,e_n$. We say that $\Gamma$ is \textit{full-rank} if $n=d$. For a full-rank lattice $\Gamma\subset\R^n$, set
$\mu(\Gamma):=\sup_{x\in\R^n}\inf_{\gamma\in\Gamma}\|x-\gamma\|$.
For $\Gamma=\Z^n$, $\mu(\Z^n)=\sqrt n/2$, by taking $x=(\frac{1}{2},\ldots,\frac{1}{2})\in \R^n$.

For a finite $A\subseteq \R^n$ we denote by $|A|$ the \textit{cardinality} of $A$. We say that $A\subset\Z^n$ generates a lattice to mean that the $\Z$-module $\langle A \rangle_\Z=\left\{\sum_{a\in A}z_a a: z_a\in\Z\right\}$ is a full-rank lattice in its real span.

\begin{lemma}\label{lem:rounding}
Let $\Gamma\subset\R^n$ be a full-rank lattice with $\mu=\mu(\Gamma)$.
Let $K,L, M\subseteq\R^n$, where $M$ is finite, and
$K\subseteq M+L$. Then there exists $\widetilde M\subseteq\Gamma$ with
$|\widetilde M|\le|M|$ and
\[
  K\;\subseteq\;\widetilde M+\bigl(L+B(\mu)\bigr).
\]
\end{lemma}

\begin{proof}
For each $m\in M$, let $\widetilde m\in\Gamma$ such that $\|m-\widetilde m\|\le\mu$, whose existence follows from the definition of $\mu$. 
Let $\widetilde M:=\left\{\widetilde m:m\in M\right\}$. For $k\in K$, write $k=m+\ell$ with
$m\in M$ and $\ell\in L$. Then
\[
k=\widetilde m+\bigl(\ell+(m-\widetilde m)\bigr),
\]
with
$(m-\widetilde m)+\ell\in B(\mu)+L$.
\end{proof}

\begin{lemma}\label{lem:absorb}
For any $K\subseteq\R^n$ and any $s\ge t\ge 0$,
\[
  (K\ominus B(s))+B(t)\;\subseteq\;K\ominus B(s-t).
\]
\end{lemma}

\begin{proof}
Let $x\in K\ominus B(s)$, so that $x+B(s)\subseteq K$. Since $B(t)+B(s-t)\subseteq B(s)$, we have
\[
  (x+B(t))+B(s-t)\subseteq x+B(s)\subseteq K,
\]
which is precisely $x+B(t)\subseteq K\ominus B(s-t)$. As $x$ was arbitrary, $(K\ominus B(s))+B(t)\subseteq K\ominus B(s-t)$.
\end{proof}

\section{Proofs of Theorems~\ref{thm:main} and~\ref{thm:finite-abelian}}\label{sec:finite}
We will need the following theorems. The first is due to
Khovanskii~\cite{Khovanskii1992}.

\begin{theorem}\label{thm:structure}
Let $A\subseteq \Z^n$ be a finite set such that  $\langle A-A\rangle_\Z=\Z^n$. Let
$P:=\mathrm{conv}(A)$. There exists a constant $f_A\ge 0$, depending only
on $A$, such that for every integer $h\ge 1$,
\begin{equation*}
  (hP\ominus B( f_A))\cap\Z^n\subseteq\;hA.
\end{equation*}
\end{theorem}
\begin{proof}
    This follows directly from Theorem~3 in \cite{Khovanskii1992}, which gives
\begin{equation*}
  \bigl\{z\in hP\cap\Z^n:\mathrm{dist}(z,\partial(hP))\ge f_A\bigr\}
  \;\subseteq\;hA,
\end{equation*}
    along with Lemma~\ref{lem:minus}, the statement follows.
\end{proof}
\begin{remark}
The constants in the above theorem can all be made effective due to \cite{GranvilleShakanWalker2023} and \cite{GranvilleSmithWalker2025}. 
\end{remark}

For any $B\subseteq \R^n$ and $X,Y\subseteq \R^n$, we say that $B$ is \textit{covered} by translates of $Y$ if $B\subseteq X+Y$ for some finite set $X\subseteq \R^n$; such $X$ is called a \textit{covering set}. 
We will need the following covering number bound (i.e. the cardinality of covering set), which is primarily based on Rogers--Zong~\cite{RogersZong1997}.

\begin{theorem}\label{thm:RZ}
For every convex body $K\subseteq\R^n$ and every $\lambda\ge 1$,
\[
  N(\lambda K, K)
  :=\min\bigl\{|M|:M\subseteq\R^n,\,\lambda K \subseteq M+ K\bigr\}
  \le C(n,\lambda)\theta(K),
\]
where $C(n,\lambda)$ is as defined in Theorem~\ref{thm:main}, and $\theta(K)$ denotes the infimum covering density of $\R^n$ with translates of $K$.

Furthermore, If $K$ is centrally symmetric (i.e. symmetric about some point $c\in \R^n$), the factor $C(n,\lambda)$ can be replaced by $(1+\lambda)^n$ for all $n\ge 1$.
\end{theorem} 
\begin{proof}
    We use the following inequality from \cite[(6)]{RogersZong1997}:
    \[N(\lambda K, K)\le \dfrac{\vol(\lambda K- K)}{\vol( K)}\theta(K),\]
     where $\vol(\cdot)$ denotes the Lebesgue volume of a given measurable set.  By mixed volume expansion,
\begin{align*}
\vol(\lambda K - K) &= \sum_{i=0}^{n}\binom{n}{i} V_i \, \lambda^i, \\
V_i &= V(K[i], (-K)[n-i]),
\end{align*}
where $V(K_1,\ldots,K_n)$ is \textit{the mixed volume} on $n$ convex bodies $K_1,\ldots,K_n$  (cf. $(5.28)$~in \cite{Schneider}). We have $V(K,\ldots,K)=\vol(K)$, and for any $i\in \N$ we write $K[i]$ to mean $K$ repeated $i$ times, i.e. $V(K[i],(-K)[n-i])$ is a shorthand for $V(\underbrace{K,\ldots,K}_{\text{$i$ times}},\underbrace{-K,\ldots,-K}_{\text{$n-i$ times}})$. Hence $V_0=V_n=\vol(K)$.

It is a classical fact of mixed volumes (see, e.g., \cite{AEFO}) that for $i \in \left\{0, 1, n-1, n\right\}$, one has $V_i \le \binom{n}{i}\vol(K)$. Therefore, for dimensions $n \le 3$, all possible indices $i \in \left\{0, \dots, n\right\}$ strictly fall into this subset. This means that for any convex body in $n \le 3$ we have:
\begin{align*}
\operatorname{vol}(\lambda K - K) 
&= \sum_{i=0}^{n}\binom{n}{i}V_i\lambda^i \\
&\le \sum_{i=0}^{n}\binom{n}{i}\left[\binom{n}{i}\operatorname{vol}(K)\right]\lambda^i \\
&= \operatorname{vol}(K) \sum_{i=0}^n \binom{n}{i}^2 \lambda^i \\
&= \operatorname{vol}(K) q_n(\lambda),
\end{align*}
where $q_n(\lambda)=\sum_{i=0}^n\binom{n}{i}^2\lambda^i$.  Substituting $q_1(\lambda)=1+\lambda$, $q_2(\lambda)=1+4\lambda+\lambda^2$, and $q_3(\lambda)=1+9\lambda+9\lambda^2+\lambda^3$, we obtain the exact polynomial bounds $C(n,\lambda) = q_n(\lambda)$ for $n \le 3$

For $i=2,\ldots,n-2$ we have the following \cite[Theorem~1.4]{AEFO}:

\begin{align*}
V_i &\le \frac{n^n}{i^i (n-i)^{n-i}} \vol(K) = \binom{n}{i} \gamma_{n,i} \vol(K), \\
\gamma_{n,i} &:= \frac{n^n}{i^i (n-i)^{n-i} \binom{n}{i}} \le \sqrt{\frac{\pi n}{2}} e^{1/12} =: c_n,
\end{align*}
    where the last bounds come from applying Robbins' bound for Stirling's approximation \cite{Robbins}, and using $i(n-i)\le n^2/4$. For $i=0,1,n-1,n$, let $\gamma_{n,i}=1\le c_n$ (as above). Consequently, for $n\ge 4$,

\begin{align*}
\frac{\vol(\lambda K - K)}{\vol(K)} \le \sum^n_{i=0} \binom{n}{i}^2 \gamma_{n,i} \lambda^i \le c_n \sum^n_{i=0}\left(\binom{n}{i}\lambda^{i/2}\right)^2\le c_n \left(\sum^n_{i=0}\binom{n}{i}\lambda^{i/2}\right)^2=c_n(1+\sqrt{\lambda})^{2n}=C(n,\lambda).
\end{align*}
    
      If $K$ is centrally symmetric, then by the translation invariance of covering numbers we may assume $K=-K$. In this case $\lambda K- K=\lambda K+ K = (1+\lambda)K$, so $\vol(\lambda K-K)=(1+\lambda)^n\vol(K)$, which gives
    $N(\lambda K,K)\le (1+\lambda)^n\theta(K)$.
\end{proof}

\begin{remark}\label{rmk}\begin{enumerate}
    \item Note that for any $v\in \R^n$, we have $N(K,L+v)=N(K,L)$, i.e. the smallest number of translates is translation invariant: if there exists $M\subseteq \R^n$ such that $K\subseteq M+L$ and $|M|=N(K,L)$ then for any $v\in \R^n$, $K\subseteq (M-v)+L+v$, where $|M-v|=|M|$.

    \item Godbersen’s conjecture (cf. \cite[Conjecture~1.1]{AEFO}) states that $V(K[i],(-K)[n-i])\le \binom{n}{i}\vol(K)$, which would give us $N(\lambda K,K)\le q_n(\lambda)\theta(K)$, for $q_n$ as defined in the proof.

    \item For a convex body $K\subseteq \R^n$, we can bound $\theta(K)$ with bounds for \textit{the translative covering density} $\vartheta(n):=\sup\{\theta(K):K\subseteq \R^n\text{ is a convex body}\}$ in $\R^n$. Specifically:
    \[\vartheta( n)\le \begin{cases}
        1&\text{ if }n=1,\\
        3/2& \text{ if }n=2,\\
         n\log n +n\log\log n +5n&\text{ if }n\ge 3,  
    \end{cases}
    \]
    where $n=2$ is from \cite{fary50} and $n\ge 3$ from \cite{Rogers1957} (while $n=1$ is trivial).
\end{enumerate}
\end{remark}

\begin{proposition}\label{prop:main}
Let $n\ge 1$ and $A\subset\Z^n$ be a finite set such that $0\in A$ and $\langle A\rangle_{\Z}=\Z^n$.
Let $P:=\conv(A)$. Let $\rho$ be the inradius of $P$, that is, $\rho:=\max\{r>0:B(z,r)\subseteq P \text{ for some }z\in \R^n\}$, and choose $y\in \R^n$ such that $B(y,\rho)\subseteq P$. Let $f_A$ be as in
Theorem~\ref{thm:structure}. Let $k \ge l \ge 1$ be integers. For every $\epsilon\in[0,1)$ satisfying
\begin{equation}\label{eq:eps}
  \epsilon\;\ge\;\frac{f_A+\sqrt n/2}{l\rho},
\end{equation}
define $\lambda:=\frac{k}{l(1-\epsilon)}$. Then there exists a set
$\widetilde M\subseteq\Z^n$ with
\[
  |\widetilde M|\le C(n, \lambda)\vartheta(n)\qquad
  \text{and} \qquad
  kA\subseteq\widetilde M+lA.
\]
If $P$ is centrally symmetric, the cardinality bound may be replaced by $|\widetilde M| \le (1+\lambda)^n \vartheta(n)$.
\end{proposition}

\begin{proof}
Set 
\[P_l^{(\epsilon)}:=(1-\epsilon)(lP- ly)+ly.\] 
Because $A$ generates $\Z^n$ as a lattice, $P$ is not contained in any lower-dimensional hyperplane, so $\rho>0$. The dilated inradius ball satisfies $l B(y,\rho)=B(l y,l\rho)\subseteq lP$. Applying Lemma~\ref{lem:homothet} to the convex body $lP$ with $t=f_A+\sqrt n/2$, and using that
\eqref{eq:eps} is precisely $\epsilon\ge t/(l\rho)$:
\begin{equation}\label{step:1}
  P_l^{(\epsilon)}\;\subseteq\;lP\ominus B\bigl(f_A+\sqrt n/2\bigr).
\end{equation}

We next compare $kP$ with $P_l^{(\epsilon)}$. Expanding the definition of $P_l^{(\epsilon)}$ gives:
\[
  P_l^{(\epsilon)} = (1-\epsilon)lP - (1-\epsilon)ly + ly = (1-\epsilon)lP + \epsilon ly.
\]
Multiplying by $\lambda$, we have:
\[
  \lambda P_l^{(\epsilon)} = \lambda(1-\epsilon)lP + \lambda\epsilon ly.
\]
Since $\lambda = \frac{k}{l(1-\epsilon)}$, we have $\lambda(1-\epsilon)l = k$. Substituting this into the first term yields:
\[
  \lambda P_l^{(\epsilon)} = kP + \lambda\epsilon ly.
\]
Thus, $kP$ is exactly a translation of $\lambda P_l^{(\epsilon)}$, specifically $kP = \lambda P_l^{(\epsilon)} - \lambda\epsilon ly$.

By Theorem~\ref{thm:RZ} and the translation-invariance of covering numbers, i.e. $(1)$ of Remark~\ref{rmk}, there exists $M \subseteq \R^n$ with
\[
  |M|\le C(n, \lambda)\,\theta\bigl(P_l^{(\epsilon)}\bigr)\le C(n, \lambda)\vartheta(n),
  \qquad kP\subseteq M+P_l^{(\epsilon)}.
\]
If $P$ is centrally symmetric about $c$, then let $y=c$. If $y_0\in \R^n$ is such that $B(y_0,\rho)\subseteq P$, then by symmetry $B(2c-y_0,\rho)\subseteq P$. By convexity, 
\[\frac{1}{2}B(y_0,\rho)+\frac{1}{2}B(2c-y_0,\rho)=B(c,\rho)\subseteq P.\] 
Hence, we attain the same inradius at $y=c$, and $lP-lc$ is symmetric about $0$, so $(1-\epsilon)(lP-lc)$ is symmetric about $0$ and $P^{(\epsilon)}_l=(1-\epsilon)(lP-lc)+lc$ is symmetric about $lc$. This allows us to apply the symmetric bound of Theorem~\ref{thm:RZ} to give us $|M| \le (1+\lambda)^n \vartheta(n)$.

Now apply Lemma~\ref{lem:rounding} with $\Gamma=\Z^n$, $\mu=\sqrt n/2$,
$K=kP$, $L=P_l^{(\epsilon)}$, producing $\widetilde M\subseteq\Z^n$
with $|\widetilde M|\le|M|$ and
\begin{equation}\label{step:3}
     kP\;\subseteq\;\widetilde M+\bigl(P_l^{(\epsilon)}+B(\sqrt n/2)\bigr).
\end{equation}
Combining (\ref{step:1}) with Lemma~\ref{lem:absorb}
($s=f_A+\sqrt n/2$, $t=\sqrt n/2$, $s-t=f_A$) we have:
\begin{equation*}\label{step:4}
  P_l^{(\epsilon)}+B(\sqrt n/2)
  \;\subseteq\;\bigl(lP\ominus B(f_A+\sqrt n/2)\bigr)+B(\sqrt n/2)
  \;\subseteq\;lP\ominus B(f_A).    
\end{equation*}
Substituting this inclusion into the right-side of~\eqref{step:3}, we get
\[kP\subseteq\widetilde M+\bigl(lP\ominus B(f_A)\bigr).\]
Now let $u\in kA\subseteq kP$. Then $u- \widetilde{m}\in lP\ominus B(f_A)$ for some
$ \widetilde{m}\in\widetilde M$. Since $u,\widetilde{m}\in\Z^n$, we also have
$u- \widetilde{m}\in\Z^n$. Theorem~\ref{thm:structure} (Khovanskii's Theorem)
implies that $u- \widetilde{m}\in lA$, whence $u\in \widetilde{m}+lA$. Since this holds for all $u \in kA$, we conclude $kA \subseteq \widetilde{M} + lA$.
\end{proof}

\begin{proof}[Proof of Theorem~\ref{thm:main}]
    Let $A\subseteq \Z^n$ be a finite non-empty set with affine span of dimension $d$.

    If $d=0$, then $A=\left\{a\right\}$ for some $a\in\Z^n$. For every $h\ge 1$,
    \[
        rhA=\left\{rha\right\}=\left\{(r-1)ha\right\}+hA.
    \]
    Hence $A$ is an asymptotic $(r,1)$-approximate group. We may therefore assume
    that $d\ge 1$.
 
    Fix any $a_0\in A$ and
    replace $A$ by $A-a_0$, which does not change its affine dimension $d$, or the size of the covering
    set, but ensures that $0\in A$. Therefore,
    \[
        \langle A\rangle_{\Z}=\langle A-A\rangle_{\Z}=:\Gamma\subseteq \R^n
    \]
    is a lattice of rank $d$.

    Choose a $\Z$-basis $b_1,\ldots,b_d$ of $\Gamma$ and let
    $\phi:\Z^d\longrightarrow\Gamma\subset\Z^n$ be the $\Z$-linear isomorphism
    $\phi(e_i)=b_i$. Set $A'=\phi^{-1}(A)\subset\Z^d$. Then
    $0\in A'$, $A'$ generates $\Z^d$ as a lattice, and, as $\phi$ is a group
    isomorphism, we have $\phi(hA')=hA$ and $\phi(rhA')=rhA$ for all $h\in\N$.

       Fix $\delta>0$. Since $C(d,\lambda)$ is a strictly increasing continuous function of $\lambda\ge 1$, there exists $r_{\delta}>r$ such that $C(d,r_\delta)=(1+\delta)C(d, r)$. For $d\ge 4$ set
    \[
 r_\delta:=\left( (1+\delta)^{1/(2d)}(1+\sqrt r)-1 \right)^2 > r.
    \]
    Let $\rho'$ be the inradius of $\conv(A')$ and $f_{A'}$ the constant of
    Theorem~\ref{thm:structure}. Choose
    \[
      h_0:=\Big\lceil\frac{f_{A'}+\sqrt d/2}{\rho'\,(1-r/r_\delta)}\Big\rceil.
    \]
    For $h\ge h_0$, set $\epsilon=(f_{A'}+\sqrt d/2)/(h\rho')$, so
    $\epsilon\le 1-r/r_\delta$ and hence $\lambda:=r/(1-\epsilon)\le r_\delta$. Because, $C(d,\lambda)$ is increasing, we have $C(d,\lambda)\le C(d,r_\delta)=(1+\delta)C(d,r)$.
    
    Apply Proposition~\ref{prop:main} to $A'$ (with ambient dimension $d$, $k=rh$ and $l=h$): there is
    $\widetilde M'\subseteq\Z^d$ with $rhA'\subseteq\widetilde M'+hA'$ and
    \[
      |\widetilde M'|\le C(d,\lambda)\vartheta(d)\le (1+\delta)C(d,r)\vartheta(d).
    \]
Let $\widetilde M:=\phi(\widetilde M')\subseteq\Z^n$. It satisfies
    $rhA=\phi(rhA')\subseteq\phi(\widetilde M')+\phi(hA')=\widetilde M+hA$ with
    $|\widetilde M|=|\widetilde M'|$. Thus, $A$ is an asymptotic $(r,\ell)$-approximate group with covering set $\widetilde M$ of size $\ell := |\widetilde M|$. Because $|\widetilde M'| \le (1+\delta)C(d,r)\vartheta(d)$, we conclude that
    \[
      \ell \le (1+\delta)C(d,r)\vartheta(d).
    \]
    The centrally symmetric case follows analogously. Indeed, translating $A$ to
    $A-a_0$ and applying $\phi^{-1}$ preserve central symmetry of the convex hull.
    Thus the homothets remain centrally symmetric, and we may invoke the symmetric
    bound from Theorem~\ref{thm:RZ}. In this case, choose
    \[
        r_\delta^{\mathrm{sym}}
        :=
        (1+\delta)^{1/d}(1+r)-1>r,
    \]
    and repeat the preceding argument with the bound
    \[
        (1+\lambda)^d\vartheta(d)
        \le
        (1+\delta)(1+r)^d\vartheta(d).
    \]
    Hence one may replace $C(d,r)$ by $(1+r)^d$ in the centrally symmetric case.

    The geometric quantities $f_{A'}$ and $\rho'$ enter only into $h_0$, not into
    the final bound for $\ell$. The choice of the basis of $\Gamma$ therefore
    affects only the threshold $h_0$, not the asymptotic covering constant.
\end{proof}

\begin{example}\label{example:sharpness}
We compare a centrally symmetric example with a non-centrally symmetric one in $\Z^2$. Let
\[
A_\square=\{(0,0),(1,0),(0,1),(1,1)\}.
\]
Then
\[
P_\square:=\operatorname{conv}(A_\square)=[0,1]^2
\]
is centrally symmetric, with centre $(1/2,1/2)$. For every $h\ge 1$,
\[
hA_\square=\{0,1,\ldots,h\}^2.
\]
Therefore
\[
rhA_\square=\{0,1,\ldots,rh\}^2.
\]
Define
\[
X^\square=\{(ih,jh):0\le i,j\le r-1\}.
\]
Then
\[
|X^\square|=r^2
\]
and
\[
rhA_\square\subseteq X^\square+hA_\square.
\]
Indeed, the square $\{0,1,\ldots,rh\}^2$ is covered by the $r^2$ translates
\[
(ih,jh)+\{0,1,\ldots,h\}^2,
\qquad 0\le i,j\le r-1.
\]
Moreover, this value is optimal in $r$: if
\[
rhA_\square\subseteq X+hA_\square,
\]
then
\[
|X|
\ge
\frac{|rhA_\square|}{|hA_\square|}
=
\frac{(rh+1)^2}{(h+1)^2}
\longrightarrow r^2
\]
as $h\to\infty$. Hence $A_\square$ is an asymptotic $(r,r^2)\text{-approximate group},$ and the constant $r^2$ is best possible asymptotically. To summarise, for $A_\square$ we have $\theta(\square)=1$, while $\vartheta(2)\le 3/2$, so the true optimal constant $r^2$ is smaller by up to a factor of $3/2$ than the bound $(1+r)^2\vartheta(2)$ produced by Theorem~\ref{thm:main}.

Now compare this with the triangular set
\[
A_\triangle=\{(0,0),(1,0),(0,1)\}.
\]
Here
\[
P_\triangle:=\operatorname{conv}(A_\triangle)
\]
is a right isosceles triangle and is not centrally symmetric. For every $h\ge 1$,
\[
hA_\triangle
=
\{(x,y)\in\N^2:x+y\le h\}.
\]
The exact minimal number of translates of $hA_\triangle$ needed to cover
\[
rhA_\triangle
=
\{(x,y)\in\N^2:x+y\le rh\}
\]
is a more delicate finite covering problem. However, the relevant planar density phenomenon is classical. F\'ary proved that among lattice coverings of the plane by translates of a planar convex domain, the minimum possible covering density is at most $3/2$, and equality occurs precisely for triangles~\cite{fary50}. Thus, for a triangle, the optimal lattice covering density is $3/2$.

In contrast, the square $P_\square$ tiles the plane by translations, so its lattice covering density is $1$. Thus the centrally symmetric square has density $1$, while the triangular non-centrally symmetric example has density $3/2$ in the lattice-covering sense. This gives a concrete geometric reason why the centrally symmetric hypothesis in Theorem~\ref{thm:main} leads to better covering bounds.
\end{example}

\subsection{Finite sets in abelian groups with torsion} 

We shall need the following lemma of Nathanson:

\begin{lemma}[{\cite[Lemma~4]{Nathanson2018}}]\label{lemma:N_torsion}
    Let $G$ be a finite group of order $k$. Then for every integer $r\ge 1 $, every non-empty subset of $G$ is an asymptotic $(r,k)$-approximate group.
\end{lemma}

\begin{proof}[Proof of Theorem~\ref{thm:finite-abelian}]
Let $\pi:H\to H/T$ be the canonical projection, and let $\Phi:=\psi\circ\pi:H\to\Z^n$. By definition, $\ker \Phi=T$ and $\Phi(A-a_0)=\overline{A}$.

If $d=0$, then $\overline{A}=\{0\}$ and hence $A-a_0\subseteq T$. Since
\[
hA=ha_0+h(A-a_0),
\qquad
rhA=rha_0+rh(A-a_0),
\]
Lemma~\ref{lemma:N_torsion} applied in the finite group $T$ gives $rh(A-a_0)\subseteq X+h(A-a_0)$ with $|X|\le |T|$ for all sufficiently large $h$. Therefore
\[
rhA\subseteq \bigl(X+(r-1)ha_0\bigr)+hA,
\]
so $A$ is an asymptotic $(r,|T|)$-approximate group.

Assume now that $d\ge 1$ and apply Theorem~\ref{thm:main} to $\overline{A}\subseteq\Z^d$. For all sufficiently large $h$, there exists $\overline{X}\subseteq\Z^d$ such that
\[
|\overline{X}|\le (1+\delta)C(d,r)\vartheta(d)
\]
and
\[
rh\overline{A}\subseteq \overline{X}+h\overline{A}.
\]
Choose a set $X_{0}\subseteq H$ containing one representative in $\Phi^{-1}(\overline{x})$ for each $\overline{x}\in\overline{X}$. Then $|X_{0}|=|\overline{X}|$, and pulling back under $\Phi$ gives
\[
rh(A-a_0)\subseteq X_{0}+T+h(A-a_0).
\]
Let $X=X_{0}+T$. Then
\[
|X|\le |X_{0}|\,|T|\le |T|(1+\delta)C(d,r)\vartheta(d)
\]
and
\[
rh(A-a_0)\subseteq X+h(A-a_0).
\]
Using
\[
rhA=rha_0+rh(A-a_0),
\qquad
hA=ha_0+h(A-a_0),
\]
we obtain
\[
rhA\subseteq \bigl(X+(r-1)ha_0\bigr)+hA.
\]
This proves the stated bound. If $\conv(\overline{A})$ is centrally symmetric, the same proof uses the symmetric bound from Theorem~\ref{thm:main}.

Finally, changing $a_0\to a_0'\in A$ translates $\overline A$ by $\psi(\pi(a_0-a_0'))\in\Z^n$; changing $\psi\to\psi'$ replaces $\overline A$ by $g(\overline A)$ for some $g\in GL_n(\Z)$. Translations and linear automorphisms both preserve $\dim\aff(\cdot)$, and preserve central symmetry since if $S=2c-S$ then $g(S)+v=2(g(c)+v)-(g(S)+v)$.
\end{proof}

\section{Semi-linear sets}\label{sec:semilinear}

We state a consequence of Theorem~\ref{thm:main} for semi-linear subsets
of $\Z^n$. The result is stated relative to a fixed semi-linear
presentation, since the bound depends only on the finite set of base points
appearing in that presentation.

Throughout this section, fix such a presentation and put
\[
  F:=\left\{a_1,\ldots,a_s\right\},\qquad
  S:=S_1+\cdots+S_s,\qquad
  b:=a_1+\cdots+a_s.
\]
The set $S$ is an additive submonoid of $\Z^n$.

We shall use the following variant of the finite-set theorem, in which the
two dilation parameters are allowed to differ by a bounded additive error.

\begin{lemma}\label{lem:shifted-finite-covering}
Let $F\subseteq\Z^n$ be finite and non-empty, and put $d_F=\dim\aff(F)$. Fix integers $r\ge 1$ and $s\ge 0$, and let $\delta>0$. If $d_F=0$, then for all $h\ge s$ there is a singleton $Y\subseteq\Z^n$ such that
\[
rhF\subseteq Y+(h-s)F.
\]
If $d_F\ge 1$, then for all sufficiently large $h$ there is a finite set $Y\subseteq\Z^n$ such that
\[
rhF\subseteq Y+(h-s)F
\]
and
\[
|Y|\le (1+\delta)C(d_F,r)\vartheta(d_F).
\]
If $\conv(F)$ is centrally symmetric, then the bound may be replaced by
\[
|Y|\le (1+\delta)(1+r)^{d_F}\vartheta(d_F).
\]
\end{lemma}

\begin{proof}
The case $d_F=0$ is immediate: if $F=\{a\}$, then
\[
rhF=\{rha\}=\{(rh-h+s)a\}+(h-s)F.
\]
Assume $d_F\ge 1$. Translate $F$ by some $a_0\in F$ and identify the lattice generated by $F-a_0$ with $\Z^{d_F}$, exactly as in the proof of Theorem~\ref{thm:main}. This does not change cardinalities of covering sets, and it only translates the final covering set.

Choose $r_\delta>r$ so that $C(d_F,r_\delta)\le (1+\delta)C(d_F,r)$; in the centrally symmetric case choose $r_\delta^{\rm sym}>r$ so that $(1+r_\delta^{\rm sym})^{d_F}\le (1+\delta)(1+r)^{d_F}$. In the proof of Theorem~\ref{thm:main}, replace the parameters $k=rh$ and $l=h$ in Proposition~\ref{prop:main} by
\[
k=rh,\qquad l=h-s.
\]
For all sufficiently large $h$, the boundary-loss parameter $\epsilon_h$ can be chosen so that
\[
\frac{rh}{(h-s)(1-\epsilon_h)}\le r_\delta
\]
(and analogously with $r_\delta^{\rm sym}$ in the centrally symmetric case). Proposition~\ref{prop:main} then gives
\[
rhF\subseteq Y+(h-s)F
\]
with the displayed bounds. Translating back to the original affine lattice gives the claimed inclusion in $\Z^n$.
\end{proof}

\begin{proof}[Proof of Theorem~\ref{thm:semilinear}]
Let
\[
  S:=S_1+\cdots+S_s,
  \qquad
  b:=a_1+\cdots+a_s.
\]

Since $A\subseteq F+S$ and $S$ is an additive monoid, for every
$h\ge 1$,
\begin{equation}\label{eq:semilinear-upper}
  rhA\subseteq rh(F+S)\subseteq rhF+S.
\end{equation}
Indeed, every positive sumset of $S$ is contained in $S$.

We next show that, for every $h\ge s$,
\begin{equation}\label{eq:semilinear-lower}
  b+(h-s)F+S\subseteq hA.
\end{equation}
Let $x\in b+(h-s)F+S$. Then
\[
  x=b+a_{i_1}+\cdots+a_{i_{h-s}}+p,
\]
where $a_{i_j}\in F$ and $p\in S$. Since $p\in S_1+\cdots+S_s$, we may
write
\[
  p=p_1+\cdots+p_s,
  \qquad p_i\in S_i.
\]
Thus
\[
  x=(a_1+p_1)+\cdots+(a_s+p_s)
    +(a_{i_1}+0)+\cdots+(a_{i_{h-s}}+0).
\]
Each summand belongs to $A$, and the number of summands is $h$. Hence
$x\in hA$, proving \eqref{eq:semilinear-lower}.

By Lemma~\ref{lem:shifted-finite-covering}, for all sufficiently large $h$ there is a finite set
$Y\subseteq\Z^n$ such that
\[
  rhF\subseteq Y+(h-s)F,
\]
and, if $d_F=0$, one may take $|Y|=1$, while if $d_F\ge 1$ then
\[
  |Y|\le (1+\delta)C(d_F,r)\vartheta(d_F),
\]
with the stated improvement when $\operatorname{conv}(F)$ is centrally symmetric.

Combining this inclusion with \eqref{eq:semilinear-upper}, we obtain
\[
  rhA
  \subseteq rhF+S
  \subseteq Y+(h-s)F+S.
\]
By \eqref{eq:semilinear-lower},
\[
  (h-s)F+S\subseteq -b+hA.
\]
Therefore
\[
  rhA\subseteq (Y-b)+hA.
\]
Setting \[ X:=Y-b, \] we have \[ rhA\subseteq X+hA. \] Moreover $|X|=|Y|$. Hence, if $d_F=0$, we may take $|X|=1$. If $d_F\ge 1$, then \[ |X|\le (1+\delta)C(d_F,r)\vartheta(d_F), \] with the stated improvement when $\conv(F)$ is centrally symmetric. This proves the theorem.
\end{proof}

\begin{remark}
The bound in Theorem~\ref{thm:semilinear} depends on the chosen
semi-linear presentation of $A$. Thus two different presentations of the
same set may yield different finite base sets $F$, and hence different
values of $d_F$. The theorem applies to every fixed presentation.
\end{remark}

\begin{example}
The affine dimension $d_F$ in Theorem~\ref{thm:semilinear} is attached to the chosen semilinear presentation, not to the underlying set alone.

For example, let
\[
A=\N\subseteq\Z.
\]
It has the one-component presentation
\[
A=0+\N,
\]
so the associated base set is
\[
F=\{0\},
\]
and hence $d_F=0$.

However, the same set also has the presentation
\[
A=(0+\N)\cup(1+\N),
\]
because $1+\N\subseteq\N$. For this presentation the base set is
\[
F'=\{0,1\},
\]
so
\[
d_{F'}=\dim\operatorname{aff}\{0,1\}=1.
\]
Thus two presentations of the same semilinear set can give different affine dimensions. In applications one may, of course, try to choose a presentation for which $d_F$ is as small as possible.
\end{example}

Analogous to the finite case, we deduce the following,

\begin{corollary}\label{cor:semilinear-abelian}
Let $G$ be an abelian group, and let $A=\bigcup_{i=1}^s(a_i+S_i)$ be a semi-linear subset of $G$, where each $S_i$ is a finitely generated additive submonoid of $G$. Fix this presentation. Let $H$ be the subgroup generated by $a_1,\ldots,a_s$ and by finite generating sets for $S_1,\ldots,S_s$, let $T:=H_{\mathrm{tors}}$, and let $\overline F$ be the image of $F:=\{a_1,\ldots,a_s\}$ in $H/T$. After choosing a $\Z$-basis of $H/T$, put $d_F:=\dim\aff(\overline F)$. Then, for every integer $r\ge 1$ and every $\delta>0$, $A$ is an asymptotic $(r,\ell)$-approximate group where, 
\begin{enumerate}
    \item If $d_F=0$, one may take $\ell\le |T|$.
    \item  If $d_F\ge 1$, one may take $\ell\le  |T|(1+\delta)C(d_F,r)\vartheta(d_F)$.
    \item  If $\conv(\overline F)$ is centrally symmetric, then $C(d_F,r)$ may be replaced by $(1+r)^{d_F}$.
\end{enumerate}  
\end{corollary}

\section*{Acknowledgements}
This research was supported through the programme ``Research in Pairs'' by the Mathematisches Forschungsinstitut Oberwolfach in 2022. P.Y. was supported by Charles University programme PRIMUS/24/SCI/010 and Czech Science Foundation, grant number 26-20514S.


\end{document}